\documentclass[12pt,nosumlimits,nonamelimits]{amsart}
\usepackage{mathtools}

\newtheorem{thmA}{Theorem}

\newtheorem*{corA}{Corollary}
\newtheorem{thm}{Theorem}
\newtheorem{prop}[thm]{Proposition}
\newtheorem{lem}[thm]{Lemma}
\theoremstyle{remark}
\newtheorem{rem}[thm]{Remark}
\newcommand{\IC}{\mathbb{C}}
\newcommand{\IM}{\mathbb{M}}
\newcommand{\IN}{\mathbb{N}}
\newcommand{\IT}{\mathbb{T}}
\newcommand{\IZ}{\mathbb{Z}}

\newcommand{\cH}{\mathcal{H}}
\newcommand{\cK}{\mathcal{K}}
\newcommand{\cL}{\mathcal{L}}

\newcommand{\cQ}{\mathcal{Q}}
\newcommand{\cR}{\mathcal{R}}
\newcommand{\cS}{\mathcal{S}}
\newcommand{\cU}{\mathcal{U}}

\newcommand{\ve}{\varepsilon}
\newcommand{\vp}{\varphi}
\newcommand{\vs}{\varsigma}

\newcommand{\id}{\mathrm{id}}
\DeclareMathOperator{\Aut}{Aut}
\DeclareMathOperator{\Int}{\smash{\overline{\mathrm{Int}}}}
\DeclareMathOperator{\Out}{Out}
\DeclareMathOperator{\Ad}{Ad}
\newcommand{\ns}{\mathop{\mathcal{S}_{\mathrm{n}}}}
\DeclareMathOperator{\Der}{Der}
\DeclareMathOperator{\ran}{ran}
\DeclareMathOperator{\supp}{supp}
\DeclareMathOperator{\dist}{d}
\DeclareMathOperator{\Ball}{Ball}

\DeclareMathOperator{\UCP}{UCP}
\DeclareMathOperator{\End}{End}
\DeclareMathOperator{\CAR}{CAR}
\newcommand{\ip}[1]{\mathopen{\langle}#1\mathclose{\rangle}}

\title[Contractibility of the automorphism group]{Contractibility 
of the automorphism group \\ of a von Neumann algebra}
\author{Narutaka Ozawa}
\address{RIMS, Kyoto University, \mbox{606-8502} Japan}
\email{narutaka@kurims.kyoto-u.ac.jp}
\thanks{The author was partially supported by JSPS KAKENHI Grant Numbers 24K00527 and 20H00114}

\date{\today}

\begin{document}
\begin{abstract}
We prove that the approximately inner automorphism group 
of a separable strongly stable von Neumann algebra 
is contractible in the u-topo\-logy. 
Thus the automorphism group of the hyperfinite type $\mathrm{III}_1$ 
factor is contractible. 
\end{abstract}
\maketitle
\section{Introduction}
In this paper, we continue Popa and Takesaki's 
study (\cite{pt}) of contractibility of the unitary group and 
the automorphism group of a von Neumann algebra. 

Let $M$ be a von Neumann algebra and $\Aut(M)$ denote 
the group of $*$-auto\-mor\-phisms on $M$, 
equipped with the u-topo\-logy, i.e., the pointwise 
convergence topology on the predual $M_*$. 
The u-topo\-logy makes $\Aut(M)$ a topological group, 
which is Polish if $M$ has separable predual. 
We are interested in the closed normal 
subgroup $\Int(M)\subset\Aut(M)$ 
of \emph{approximately inner} automorphisms, 
which is the closure of the subgroup 
$\mathrm{Int}(M)$ of inner automorphisms. 
The von Neumann algebra $M$ is said to 
be \emph{strongly stable} (or \emph{McDuff}) 
if $M\cong M\otimes\cR$, 
where $\cR$ denotes the hyperfinite (or AFD) factor 
of type $\mathrm{II}_1$ with separable predual. 
The following theorem generalizes Popa 
and Takesaki's theorem (\cite{pt}). 

\begin{thmA}\label{thmA}
Let $M$ be a strongly stable von Neumann algebra 
with separable predual. 
Then $\Int(M)$ is contractible. 
\end{thmA}

The case for the hyperfinite factor $\cR_{\mathrm{III}_1}$ 
of type $\mathrm{III}_1$ with separable predual 
has attracted considerable attention because 
of applications to geometric topology and 
mathematical physics (see e.g., \cite{st}). 
The factor $\cR_{\mathrm{III}_1}$ is strongly stable 
(and so are all separable hyperfinite type $\mathrm{II}$ 
and type $\mathrm{III}$ von Neumann algebras, 
see \cite{takesakiIII}). 
That $\Aut(\cR_{\mathrm{III}_1})=\Int(\cR_{\mathrm{III}_1})$ 
is proved by Kawahigashi, Sutherland, and Takesaki (\cite{kst}).

\begin{corA}\label{corB}
The automorphism group $\Aut(\cR_{\mathrm{III}_1})$ 
of $\cR_{\mathrm{III}_1}$ is contractible. 
\end{corA}

We denote by $\cU(M)$ the unitary group of $M$ 
equipped with the ultrastrong topology. 
Let $N\subset M$ be an inclusion of von Neumann algebras 
(which is always assumed to be unital). 
The unitary group $\cU(N)$ of $N$ will be identified 
with its image in $\cU(M)$, as they are canonically 
homeomorphically isomorphic. 
The following is asserted in \cite{pt}, 
but the proof contains a gap, as explained in Section~\ref{sec:thmc}. 
We bridge the gap on this occasion. 

\begin{thmA}\label{thmB}
Let $N\subset M$ be an inclusion of $\sigma$-finite 
von Neumann algebras and assume 
that $N$ is strongly stable. 
Then the quotient map from $\mathcal{U}(M)$ 
onto $\mathcal{U}(M)/\mathcal{U}(N)$ 
admits a continuous cross section. 
Equivalently, there is a continuous equivariant retraction 
from $\cU(M)$ onto $\cU(N)$.
\end{thmA}

\subsection*{Acknowledgment} 
The author is grateful for Yusuke Isono for explaining 
him the contents of \cite{pt} and for pointing out an error 
in an earlier draft, for Masaki Izumi for drawing his attention 
to \cite{dp}, and for the referee for helpful suggestion concerning 
the exposition.
\section{Preliminary for Proof of Theorem~\ref{thmA}}\label{sec:prelim}
Let $M$ be a von Neumann algebra with a faithful normal 
state $\omega$. 
The corresponding $2$-norm on $M$ is given by 
$\|x\|_2\coloneqq\omega(x^*x)^{1/2}$. 
The topology induced by the $2$-norm coincides 
with the ultrastrong topology on bounded subsets of $M$. 
It also coincides with the ultraweak topology on 
the unitary group $\cU(M)$. 
For $a\in M$, we define $a\omega\in M_*$ by 
$(a\omega)(x)\coloneqq\omega(xa)$. 
One has $\|a\omega\|\le\|a\|_2$ 
by the Cauchy--Schwarz inequality.

Let $\UCP_\sigma(M,N)$ denote the set of ultraweakly 
continuous unital completely positive maps from a von Neumann 
algebra $M$ into $N$. Note that $\Aut(M)\subset\UCP_\sigma(M,M)$. 
Let $\ns(N)$ denote the set of normal states on $N$. 
For $\alpha\in\UCP_\sigma(M,N)$, we define its pre-conjugate 
$\alpha^*\colon N_*\to M_*$ by 
$\alpha^*(\vp)=\vp\circ\alpha$. 
Note that $(\alpha\beta)^*=\beta^*\alpha^*$.
We extend the definition of the u-topo\-logy on $\Aut(M)$ and 
define the \emph{u-topo\-logy} on $\UCP_\sigma(M,N)$ to be 
the pointwise norm convergence topology of the pre-conjugates. 
Namely, the u-topo\-logy is induced by the pseudo-metrics 
$\{ \dist_\vp : \vp\in \ns(N)\}$, where 
\[
\dist_\vp(\alpha,\beta) \coloneqq \|(\alpha^*-\beta^*)(\vp)\|
 =\|\vp\circ\alpha-\vp\circ\beta\|. 
\]
We also write $\dist_F\coloneqq\max_{\vp\in F}\dist_\vp$ 
for a finite subset $F\subset \ns(N)$. 
We note that the composition is jointly u-con\-ti\-nuous 
(because all pre-conjugates are norm contractive) 
and that the u-topo\-logy is compatible with the tensor product of von Neumann 
algebras (because the algebraic tensor product $(M_1)_*\otimes (M_2)_*$ is 
dense in $(M_1\otimes M_2)_*$). 
Moreover, the pseudo-metric $\dist_\vp$ is right invariant on $\Aut(M)$, i.e., 
$\dist_\vp(\alpha,\beta)=\dist_\vp(\id,\beta\alpha^{-1})$ 
for every $\alpha,\beta\in\Aut(M)$, and the homomorphism 
$\cU(M)\ni u\mapsto\Ad u\in \Aut(M)$ is continuous. 
Here $\Ad u$ is the inner automorphism 
defined by $(\Ad u)(x)=uxu^*$. 

We consider a type $\mathrm{II}_1$ factor $N$ 
and denote by $\End(N)$ the set 
of unital $*$-endo\-mor\-phisms on $N$. 
The unique tracial state 
is denoted by $\tau$, or $\tau_N$ to emphasize $N$. 
We recall that the \emph{p-topo\-logy} on $\End(N)$ is 
the pointwise ultraweak convergence topology. 
It coincides with the pointwise 2-norm convergence topology, 
where the 2-norm is taken w.r.t.\ the trace $\tau$.  
Every $\psi\in\End(N)$ is trace-preserving and thus 
isometric w.r.t.\ the 2-norm. 
The p-topo\-logy is weaker than the u-topo\-logy, 
but they coincide on $\Aut(N)$. 

\begin{lem}\label{lem:p}
Let $N$ be a type $\mathrm{II}_1$ factor. 
Suppose that $\psi_n\in\Aut(N)$ converge 
to $\psi\in\End(N)$ in the p-topo\-logy, 
and denote by $E$ the unique trace-preserving conditional 
expectation from $N$ onto $\psi(N)$. 
Then, $E\psi_n\to\psi$ and $\psi_n^{-1}\to\psi^{-1}E$ in the u-topo\-logy. 
\end{lem}
\begin{proof}
Let $a\in N$ be such that $a\geq0$ and $\tau(a)=1$. 
Then since 
\[
\| (\psi^{-1}E)(a) - (\psi_n^{-1}E)(a) \|_2 = \| \psi_n\psi^{-1} (E(a)) - E(a) \|_2\to0,
\]
one has 
\[
(E\psi_n)^*(a\tau)=(\psi_n^{-1}E)(a)\tau \to (\psi^{-1}E)(a)\tau=(E\psi)^*(a\tau)
\]
in norm. 
Since the states of the form $a\tau$ are norm dense in $\ns(N)$, 
the first assertion follows. The proof of the second is similar. 
\end{proof}

We collect a few well-known facts about 
the hyperfinite type $\mathrm{II}_1$ factor $\cR$. 
See \cite{takesakiIII} for general information.
We set $\cR_\infty\coloneqq\bigotimes_{n=1}^\infty\cR$ and 
write $\cR_n\coloneqq\bigotimes_{m=1}^n\cR$ the first $n$ tensor product. 
Hence for $\cR_n^{\mathrm{c}}\coloneqq\bigotimes_{m=n+1}^\infty\cR$, 
one has $\cR_\infty=\cR_n\otimes\cR_n^{\mathrm{c}}$.
Note that $\cR_n\cong\cR$ for every $n\in\IN\cup\{\infty\}$. 
The factor $\cR$ is 
\emph{strongly self-absorbing} in the following sense. 

\begin{lem}\label{lem:ssa}
Let's write $\cR=\cQ\otimes\cS$, 
where $\cQ$ and $\cS$ are copies of $\cR$.  
Then there are $*$-iso\-mor\-phisms 
$\rho_{\cQ} \colon\cQ\to\cR$ and $\rho_{\cS}\colon \cS\to\cR$ 
and a p-con\-ti\-nuous map 
$\psi\colon[0,1]\to\UCP_\sigma(\cR,\cR\otimes\cR)$ 
such that $\psi_0(x)=x\otimes 1$ for $x\in\cR$, 
$\psi_t$ are surjective $*$-iso\-mor\-phisms for $t>0$, 
and $\psi_1(x\otimes y)=\rho_{\cQ}(x)\otimes\rho_{\cS}(y)$ 
for $x\in\cQ$ and $y\in\cS$. 
The map $\psi$ satisfies 
$(\id_{\cR}\times\tau_{\cR})\psi_t\to\id_{\cR}$ 
and $\psi_t^{-1}\to\id_{\cR}\times\tau_{\cR}$ 
as $t\to0$ in the u-topo\-logy, 
where $\id_{\cR}\times\tau_{\cR}$ is the slice map 
given by $(\id_{\cR}\times\tau_{\cR})(x\otimes y)=\tau(y)x$. 
\end{lem}
\begin{proof}
For notational convenience, we replace $\cR=\cQ\otimes\cS$ 
with $\cR_\infty = \cR_1\otimes\cR_1^{\mathrm{c}}$. 
Moreover we work with 
$\UCP_\sigma(\cR_\infty,\cR\otimes\cR_\infty)$ 
and swap left with right 
and $t\colon 1\mapsto0$ with $s\colon1\mapsto\infty$. 
There is a u-continuous map 
$\sigma\colon [0,1]\to\Aut(\cR\otimes\cR)$ 
that connects $\sigma_0\coloneqq\id$ to the flip automorphism 
$\sigma_1$, given by $x\otimes y\mapsto y\otimes x$.
The map $\sigma$ can be constructed from that for 
the $2$-by-$2$ matrix algebra $\IM_2$ 
via the isomorphism $\cR\cong\bigotimes_{\IN}\IM_2$. 
We write $\sigma^{(n)}_t\in\Aut(\cR_\infty)$ 
the copy of $\sigma_t$ applied 
at the $\{n,n+1\}$-th tensor product component. 
For $n\in\IN$ and $t\in[0,1]$, 
we define the $*$-iso\-mor\-phism 
$\psi_{n+t}\colon \cR_\infty\to\cR\otimes\cR_\infty$ by 
$\psi_{n+t}=\psi_n\sigma^{(n)}_t$, where
\[
\psi_n(x_1\otimes x_2\otimes\cdots)
 \coloneqq x_n \otimes (x_1\otimes\cdots\otimes x_{n-1}\otimes x_{n+1}\otimes\cdots).
\]
One has $\psi_1(x\otimes y)=\rho(x) \otimes \rho'(y)$ for 
$x\in \cR_1$ and $y\in\cR_1^{\mathrm{c}}$, where 
$\rho\colon\cR_1\to\cR$ is the canonical identification and 
$\rho'\colon\cR_1^{\mathrm{c}}\to\cR_\infty$ 
is the shift isomorphism. 
The map $\psi$ is u-con\-ti\-nuous on $[1,\infty)$, and 
since $\bigcup_n \cR_n\otimes\IC1_{\cR_n^\mathrm{c}}\subset\cR_\infty$ 
is ultrastrongly dense, one has 
$\psi_s\to1\otimes\id_{\cR_\infty}$ as $s\to\infty$ 
in the p-topo\-logy 
(NB: but not in the u-topo\-logy).
That $\psi$ satisfies the last statement 
follows from Lemma~\ref{lem:p}, as 
the conditional expectation $E$ 
onto $\ran\psi_0 = \cR\otimes\IC1$ 
is $\id\otimes\tau$ and 
hence $\psi_0^{-1}E=\id\times\tau$. 
\end{proof}

\begin{lem}\label{lem:martingale}
The trace-preserving conditional expectations 
$E_n\coloneqq\id_{\cR_n}\otimes\tau_{\cR_n^{\mathrm{c}}}$ 
from $\cR_\infty$ onto $\cR_n\otimes\IC1_{\cR_n^{\mathrm{c}}}$ 
converge to $\id_{\cR_\infty}$ in the u-topo\-logy. 
\end{lem}
\begin{proof}
One has $E_n^*(a\tau)=E_n(a)\tau\to a\tau$ for every 
$a\in \bigcup_n \cR_n\otimes\IC1_{\cR_n^{\mathrm{c}}}$ 
with $a\geq0$ and $\tau(a)=1$.
Since the states of the form $a\tau$ are dense 
in $\ns(\cR_\infty)$, we are done.
\end{proof}

A unitary element $u\in\cU(M)$ is said to be 
\emph{$\omega$-Haar} (or simply \emph{Haar}) if 
$\omega(u^n)=0$ for all $n\neq0$. 
Let $u=\int_\IT z\,dE_u(z)$ be the spectral resolution of $u$, 
where  $E_u$ is the spectral measure on $\IT\coloneqq\{ z : |z|=1\}$. 
Then $u$ is $\omega$-Haar if and only if the probability measure 
$\omega\circ E_u$ coincides with the Haar (Lebesgue) measure $\lambda$ on $\IT$. 
We denote by $\log$ the branch of 
the logarithm on $\IT$ that takes values in $[-i\pi,i\pi)$. 

\begin{lem}\label{lem:haar}
Let $M$ be a von Neumann algebra with 
a faithful normal state $\omega$ and 
$W\subset\cU(M)$ be a subset consisting entirely 
of $\omega$-Haar unitary elements. 
Then, 
\[
h\colon W\times[0,1] \ni(u,t) \mapsto \exp(t\log u)\in\cU(M)
\]
is continuous and satisfies $h(u,0)=1$ and $h(u,1)=u$ for $u\in W$.
Moreover for every $\ve>0$ there is $\delta>0$ 
(independent of $\omega$ as long as $u$ is $\omega$-Haar) 
that satisfies the following property: 
If $\dist_\omega(\id,\Ad u)<\delta$, 
then $\dist_\omega(\id,\Ad h(u,t))<\ve$. 
\end{lem}
\begin{proof}
For $\kappa > 0$, consider the piecewise linear function 
$g_\kappa$ from $[-\pi,\pi]$ to $[-\pi+\kappa,\pi-\kappa]$ 
that linearly connects $g_\kappa(-\pi)=0$, 
$g_\kappa(-\pi+\kappa)=-\pi+\kappa$, 
$g_\kappa(\pi-\kappa)=\pi-\kappa$, 
and $g_\kappa(\pi)=0$.
Then 
$f_\kappa(\exp(i\theta))\coloneqq\exp( i g_\kappa(\theta))$, 
$\theta\in[-\pi,\pi]$, 
defines a continuous unitary function $f_\kappa$ on $\IT$. 
Since $\log$ is continuous on the range of $f_\kappa$, the map 
$h_\kappa(u,t) \coloneqq \exp(t\log f_\kappa(u))$ is continuous 
on $W\times[0,1]$. 
Since every $u\in W$ is $\omega$-Haar, one has 
\[
\| h(u,t) - h_\kappa(u,t) \|_2^2 
=\int_{\IT}|\exp( t \log z ) - \exp( t\log f_\kappa(z) )|^2\,d\lambda(z)
\le 4\frac{\kappa}{\pi}.
\]
It follows that $h_\kappa\to h$ uniformly as $\kappa\to0$ 
and so $h$ is continuous. 

For the second assertion, let $\ve>0$ be given.
For $a\in M$, we define $\Der a$ on $M$ 
by $(\Der a)(x)=ax-xa$. 
Note that $\Der$ is linear, $\|\omega\circ\Der a\|\le2\|a\|_2$ 
for normal $a$, 
and $\|\omega\circ\Der v\|=\dist_\omega(\id,\Ad v)$ 
for $v\in\cU(M)$. 
We take $\kappa\coloneqq\ve^2/2$ and 
consider $h_\kappa$ as above. 
Thus $\| h(u,t)-h_\kappa(u,t)\|_2<\ve$ for $u\in W$.
Take a polynomial approximation $p_\kappa\in\IZ[z,z^{-1},t]$ 
that satisfies
$| p_\kappa(z,t) - \exp(t\log f_\kappa(z))|<\ve$ 
on $\IT\times[0,1]$. 
Since $\|\omega\circ\Der u^n\|\le|n|\|\omega\circ\Der u\|$ 
for every $n\in\IZ$, there is $\delta>0$ 
such that $\|\omega\circ\Der u\|<\delta$ 
implies $\| \omega\circ\Der p_\kappa(u,t) \|<\ve$ for all $t\in[0,1]$. 
It follows that for $(u,t)\in W\times[0,1]$, that 
$\|\omega\circ\Der u\|<\delta$ implies 
\[
\|\omega\circ\Der h\|
\le 2\| h - h_\kappa\|_2
 + 2\|h_\kappa-p_\kappa \|_2
 + \| \omega\circ\Der p_\kappa \|
 < 5\ve, 
\]
where we omitted writing $(u,t)$. 
This proves the second assertion. 
\end{proof}

\section{Proof of Theorem~\ref{thmA}}
The strategy of the proof of Theorem~\ref{thmA} 
is similar to that for Theorem~4 
in \cite{pt}, but the cross section 
method is replaced with a plain convexity argument. 
Dadarlat and Pennig's trick provides a room 
for the convexity argument to work.  

\begin{prop}[cf.\ Theorem 2.3 in \cite{dp}]\label{prop:dp}
Let $M_0$ be a strongly stable von Neumann algebra, 
and $M\coloneqq M_0\otimes\cR$. 
Then there is a continuous map 
$H\colon\Aut(M) \times[0,1] \to \Aut(M)$ 
such that $H(\alpha,0)=\alpha$ and $H(\alpha,1)\in \Aut(M_0)\otimes\id_{\cR}$. 
Moreover, $H$ maps $\Int(M)\times[0,1]$ into $\Int(M)$.
\end{prop}
\begin{proof}
Let's write $\cR=\cQ\otimes\cS$ and 
take $\psi_t\colon\cR\to\cR\otimes\cR$ 
as in Lemma~\ref{lem:ssa}. 
In particular, 
$\psi_1(x\otimes y)=\rho_{\cQ}(x)\otimes\rho_{\cS}(y)$ 
and 
$(\id_{\cR}\times\tau_{\cR})\psi_t\to\id_{\cR}$ 
and $\psi_t^{-1}\to\id_{\cR}\times\tau_{\cR}$ 
as $t\to0$. 
We consider $M_0 \coloneqq M_{00}\otimes\cQ$ and 
$M\coloneqq M_0\otimes\cS=M_{00}\otimes\cR$. 
We define $H\colon \Aut(M)\times[0,1]\to \Aut(M)$ by 
$H(\alpha,0)=\alpha$ and 
\[
H(\alpha,t)=(\id_{M_{00}}\otimes\psi_t^{-1})(\alpha\otimes\id_{\cR})(\id_{M_{00}}\otimes\psi_t)
\] 
for $t>0$. Then $H$ is clearly continuous for $t>0$. 
If $\alpha_n\to\alpha$ and $t_n\to0$, then for every $\vp\in\ns(M)$ 
one has 
\begin{align*}
\lim_n H(\alpha_n,t_n)^*(\vp)
 &= \lim_n (\id_{M_{00}}\otimes\psi_{t_n})^*(\alpha_n\otimes\id_{\cR})^*(\id_{M_{00}}\otimes\psi_{t_n}^{-1})^*(\vp)\\
 &= \lim_n (\id_{M_{00}}\otimes\psi_{t_n})^*(\alpha\otimes\id_{\cR})^*(\id_M\times \tau_{\cR})^*(\vp)\\
 &=\lim_n (\id_{M_{00}}\otimes\psi_{t_n})^*(\id_M\times \tau_{\cR})^*\alpha^*(\vp)\\
 &=\alpha^*(\vp).
\end{align*}
Here we have used 
$\id_{M_{00}}\otimes(\id_{\cR}\times \tau_{\cR})=\id_M\times\tau_{\cR}$ and 
$(\id_M\times\tau_{\cR})(\alpha\otimes\id_{\cR})=\alpha(\id_M\times\tau_{\cR})$. 
This proves continuity of $H$. 
Finally, observe that 
\[
H(\alpha,1)=((\id_{M_{00}}\otimes\rho_{\cQ})^{-1}\alpha(\id_{M_{00}}\otimes\rho_{\cQ})) \otimes\id_{\cS}
\in\Aut(M_0)\otimes\id_{\cS}.
\]
That $H$ keeps the (approximately) inner automorphism 
group invariant is obvious. 
\end{proof}

\begin{lem}\label{lem:firstselection}
Let $M=M_0\otimes \cR$ be a strongly stable 
von Neumann algebra with separable predual. 
Let $F_0\subset\ns(M_0)$ be a finite subset and $\ve>0$. 
Put $F\coloneqq F_0\otimes\tau\subset\ns(M)$. 
Then, there is a continuous map 
$u\colon \Int(M_0)\to\cU(M)$ that satisfies
$\dist_F(\alpha\otimes\id,\Ad u(\alpha)) < \ve$ for all $\alpha$.
\end{lem}

\begin{proof}
We consider the open cover $\{ W_u : u\in\cU(M_0)\}$ for $\Int(M_0)$ given by 
\[
W_u\coloneqq\{ \alpha\in\Int(M_0) : \dist_{F_0}(\alpha,\Ad u)<\ve\}.
\]
Take a partition of unity $\{\psi_i\}_{i=1}^\infty$ for $\Int(M_0)$ 
subordinated by $\{ W_u\}$. 
For each $i$, take $u_i\in\cU(M_0)$ such that $\supp\psi_i\subset W_{u_i}$. 
We define $p_i(\alpha)$ to be the orthogonal projection 
in $L^\infty[0,1]\subset \cR$ that 
corresponds to the interval 
$[\sum_{j<i}\psi_j(\alpha), \sum_{j\le i}\psi_j(\alpha))$. 
Note that the defining sum is a locally finite sum and that 
the maps $\alpha\mapsto p_i(\alpha)$ are ultrastrongly continuous. 
We define a continuous map $u\colon \Int(M_0)\to \cU(M)$ by 
$u(\alpha) \coloneqq \sum_i u_i \otimes p_i(\alpha)$.
For every $\vp\in F_0$, one has 
\begin{align*}
(\vp \otimes \tau)\circ\Ad u(\alpha) 
 = \sum_i (\vp\circ\Ad u_i) \otimes (p_i(\alpha)\tau)
 \approx_\ve(\vp\circ\alpha)\otimes\tau
\end{align*}
since $\|\vp\circ\Ad u_i-\vp\circ\alpha\|<\ve$ for every $i$ and 
$\sum_i \|p_i(\alpha)\tau\|=\sum_i \tau(p_i(\alpha))=1$. 
\end{proof}

\begin{lem}\label{lem:bridge}
Let $M=M_0\otimes \cR$ be a strongly stable von Neumann algebra 
with separable predual. 
Let $F_0\subset \ns(M_0)$ be a finite subset and $\ve>0$. 
Put $F\coloneqq F_0\otimes\tau\subset\ns(M)$.
Then there are $\delta>0$ and a continuous map 
$h\colon \cU(M_0)\times[0,1]\to\cU(M)$ 
such that $h(u,0)=u\otimes1$ and $h(u,1)=1$ for $u\in\cU(M_0)$ 
and moreover that 
$\dist_F(\id,\Ad h(u,t))<\ve$ for all $t\in[0,1]$ 
if $u$ satisfies $\dist_{F_0}(\id, \Ad u)<\delta$. 
\end{lem}

\begin{proof}
Fix a $\tau$-Haar unitary element $w$ in $\cR$ 
and define a continuous map 
$h\colon \cU(M_0)\times[0,1/2]\to\cU(M)$ by 
$h(u,t)\coloneqq u\otimes \exp(2t\log w)$. 
Note that $h(u,0)=u\otimes1$, 
$h(u,1/2)=u\otimes w$, and 
$\dist_F(\id,\Ad h(u,t))=\dist_{F_0}(\id,\Ad u)$ for every $u$. 
Since all elements in $\cU(M_0)\otimes w$ 
are $\vp$-Haar for $\vp\in F$, 
Lemma~\ref{lem:haar} (we may assume that 
every $\vp\in F$ is faithful) implies that the map 
$h\colon \cU(M_0)\times[1/2,1]\to\cU(M)$ defined by 
$h(u,t)=\exp( 2(1-t)\log(u\otimes w) )$ is continuous and 
satisfies 
$\dist_F(\id,\Ad h(u,t))<\ve$ 
if $\dist_{F_0}(\id, \Ad u)<\delta$.
\end{proof}

\begin{proof}[Proof of Theorem~\ref{thmA}]
Let $M_0$ be a strongly stable von Neumann algebra 
with separable predual, and write 
$M\coloneqq M_0\otimes\cR_\infty$ and 
$M_n\coloneqq M_0\otimes\cR_n$ (see Section~\ref{sec:prelim} 
for the notation). 
Since $\id_{M_0}\otimes E_n\to\id_M$ in the u-topo\-logy
by Lemma~\ref{lem:martingale}, 
we can find an increasing sequence 
$F_1\subset F_2\subset\cdots\subset \ns(M)$ 
of finite subsets such that 
$\bigcup_n F_n$ is dense in $\ns(M)$ and that 
$F_n=F_n^0\otimes\tau_{\cR_{n-1}^{\mathrm{c}}}$ 
for some $F_n^0\subset \ns(M_{n-1})$. 
We identify each $M_n$ with $M_n\otimes\IC1_{\cR_n^{\mathrm{c}}}\subset M$ 
and omit writing $\otimes 1$. 

By Lemma~\ref{lem:bridge},
for every $n\in\{0\}\cup\IN$, there are 
continuous map $h_n\colon \cU(M_{n+1})\times[0,1]\to\cU(M)$ 
and $\delta_n>0$ such that 
$h_n(v,0)=1$ and $h_n(v,1)=v$ for $v\in\cU(M_{n+1})$; 
and that if $v$ is such that $\dist_{F_n}(\id,\Ad v)<\delta_n$, 
then $\dist_{F_n}(\id,\Ad h_n(v,t))<n^{-1}$. 
The last condition is vacant for $n=0$. 
We may assume that $\delta_0=5$ and $\delta_n\searrow0$. 

We set $u_0(\alpha) \coloneqq 1$ for all $\alpha$. 
By Lemma~\ref{lem:firstselection},
for every $n\in\IN$, there is 
a continuous map $u_n\colon \Int(M_0)\to \cU(M_n)$
such that 
$\dist_{F_n}(\alpha\otimes\id,\Ad u_n(\alpha))<\delta_n/2$ for 
every $\alpha$.  

Then for every $n\in\{0\}\cup\IN$ one has 
\[
\dist_{F_n}(\id,\Ad u_{n+1}(\alpha)u_n(\alpha)^*)
 =\dist_{F_n}(\Ad u_n(\alpha),\Ad u_{n+1}(\alpha))
 <\delta_n.
\]
For $n\in\{0\}\cup\IN$ and $t\in[0,1)$, put 
$u(\alpha,n+t)\coloneqq h_n(u_{n+1}(\alpha)u_n(\alpha)^*,t)u_n(\alpha)$.
Then 
\[
\dist_{F_n}(\Ad u_n(\alpha), \Ad u(\alpha,n+t))
=\dist_{F_n}(\id, \Ad h_n(u_{n+1}(\alpha)u_n(\alpha)^*,t))<n^{-1}
\]
for all $n$ and $t$. 
Since $(\alpha,s)\mapsto u(\alpha,s)$ is continuous, so is 
$(\alpha,s)\mapsto\Ad u(\alpha,s)\in\Int(M)$. 
One has $\Ad u(\alpha,0)=\id_M$ and 
$\Ad u(\alpha,s)\to\alpha\otimes\id$ as $s\to\infty$, 
since for each $m$ 
\[
\lim_s \dist_{F_m}(\alpha\otimes\id,\Ad u(\alpha,s)) 
= \lim_n \dist_{F_m}(\alpha\otimes\id,\Ad u_n(\alpha)) = 0.
\]
By Proposition~\ref{prop:dp}, $\Int(M)$ contracts to 
$\Int(M_0)\otimes\id_{\cR_\infty}$, 
and $\Int(M_0)\otimes\id_{\cR_\infty}$ contracts 
to $\{\id_M\}$ by the above. 
\end{proof}
\begin{rem}
For the free group factor $N\coloneqq\cL F_\infty$ 
of countably infinite rank (see Section XIV.3 in \cite{takesakiIII}), 
$\Int(N)=\mathrm{Int}(N)\cong\cU(N)/\IT$ is not contractible; 
In fact, it is a model for the Eilenberg–MacLane 
space $K(\IZ,2)$, as $\cU(N)$ is contractible 
by \cite{pt}. 
Incidentally, the Polish group 
$\Out(N)\coloneqq\Aut(N)/\mathrm{Int}(N)$ 
is a model for $K(\IZ,3)$. 
That $\Aut(N)$ is contractible follows from 
Dadarlat and Pennig's argument (\cite{dp}) 
as adapted to the free product setting: 
the free flip on $N*N$ is path-connected 
to the identity automorphism 
(e.g., via Voiculescu's free gaussian functor) 
and $N\cong N^{\ast\infty}$. 
\end{rem}
%
%
%
\section{Preliminary for Proof of Theorem~\ref{thmB}}\label{sec:prelim2}
Let $H\subset G$ be topological groups (with $H$ closed) 
and $Q\colon G\to G/H$ denote the quotient map. 
A \emph{cross section} for $Q$ 
is a map $\vs\colon G/H\to G$ 
such that $Q\circ\vs=\id_{G/H}$. 
We say a cross section $\vs$ is \emph{unital} if 
it satisfies $\vs(H)=1$. 
Every cross section $\vs$ is made unital 
by multiplying $\vs(H)^{-1}$ from the right. 
A \emph{retraction} of $G$ to $H$ is a map 
$\rho\colon G\to H$ such that $\rho|_H=\id_H$. 
A retraction $\rho$ is \emph{$H$-equivariant} 
(or simply \emph{equivariant}) if 
it satisfies 
$\rho(hg)=h\rho(g)$ for all $h\in H$ and $g\in G$. 
Let's recall two facts. 
The quotient map $Q$ is open, 
because $Q^{-1}(Q(W))=\bigcup_{h\in H}Wh$ 
is open whenever $W\subset G$ is. 
The correspondence $\rho_\vs(g)\coloneqq g\vs(g^{-1}H)$ 
and $\vs_\rho(gH)\coloneqq g\rho(g^{-1})$ is a bijection 
between unital continuous cross sections $\vs$ for $Q$ and 
equivariant continuous retractions $\rho$ from $G$ onto $H$.

Recall that the CAR algebra $\CAR(\cH)$ over 
a (separable) Hilbert space $\cH$ 
is a unital $\mathrm{C}^*$-algebra together with 
a linear (and isometric) map 
$a\colon\cH\to\CAR(\cH)$ that satisfies 
the canonical anti-commutation relation: 
\begin{align*}
a(\xi)^*a(\eta)+a(\eta)a(\xi)^* &= \ip{\eta,\xi}1 \\
a(\xi)a(\eta)+a(\eta)a(\xi) &= 0.
\end{align*}
Recall that $\CAR(\cH)\cong\IM_{2^d}$ 
for $d=\dim\cH \in\IN\cup\{\infty\}$ and 
$\CAR(\cH)$ has a unique tracial state $\tau$, 
which satisfies $\tau(a(\xi)^*a(\eta))=\frac{1}{2}\ip{\eta,\xi}$. 
Recall also that any nonzero vector $\xi\in\cH$ generates 
an isomorphic copy of $\IM_2$ in $\CAR(\cH)$. 
Indeed, $v\coloneqq a(\xi)$ for a unit vector $\xi$ satisfies $v^2=0$ and 
\[
(v^*v)^2=v^*(vv^*)v=v^*(1-v^*v)v=v^*v.
\]
So, $v$ is a partial isometry such that $v^*v+vv^*=1$.
If $\xi\perp\eta$, then 
$a(\xi)$ commutes with $a(\eta)^*a(\eta)$ because 
$a(\xi)$ anti-commutes with both $a(\eta)$ and $a(\eta)^*$. 
Hence, if $\cK\subset\cH$ and $\eta\in\cK^\perp$, then 
$a(\eta)^*a(\eta)\in \CAR(\cK)'\cap\CAR(\cH)$. 

Now set $\cH\coloneqq L^2[0,1]$ and 
consider the SOT-continuous semigroup of isometries 
$(V_t)_{t\geq0}$ on $L^2[0,1]$ given by 
\[
(V_t\xi)(r)\coloneqq 1_{[0,\exp(-t)]}(r) \exp(t/2)\xi(\exp(t)r).
\]
It induces the continuous semigroup of $*$-endo\-mor\-phisms  
$\theta_t$ on $\CAR(\cH)$ by $\theta_t(a(\xi))\coloneqq a(V_t\xi)$. 
We also consider an SOT-continuous family $\{W_t\}_{t>0}$ 
of isometric embedding of $\ell_2$ into $L^2[\exp(-t),\exp(-t/2)]$. 
It induces a continuous unital embedding $\phi_t$ of 
$\CAR(\ell_2)$ into $\CAR(\cH)$. 
The above construction lifts to the completion of $\CAR(\cH)$ 
w.r.t.\ the trace $\tau$, which is isomorphic to 
the hyperfinite type $\mathrm{II}_1$ factor $\cR$. 

\begin{lem}\label{lem:car}
There is a continuous family of trace-preserving $*$-endo\-mor\-phisms 
$(\theta_t)_{t\geq0}$ on $\cR$ 
such that $\theta_0=\id$ and 
that $\theta_t(\cR)'\cap\cR$ is diffuse for $t>0$. 
Moreover, there is a continuous family 
$(\phi_t)_{t>0}$ of trace-preserving embeddings of 
$L^\infty[0,1]$ into $\cR$ such that 
$\phi_t(L^\infty[0,1])\subset\theta_t(\cR)'\cap\theta_{t/2}(\cR)$ for every $t>0$. 
\end{lem}

Note that there is a canonical trace-preserving isomorphism 
\[
\theta_t(\cR)\vee (\theta_t(\cR)'\cap\cR)
 \cong \theta_t(\cR)\otimes(\theta_t(\cR)'\cap\cR)
\]
since $\theta_t(\cR)$ is a type $\mathrm{II}_1$ factor. 
Lemma~\ref{lem:car} can be used to reprove 
Popa and Takesaki's theorem (\cite{pt}) that 
the unitary group of a strongly stable von Neumann 
algebra is contractible. 
It is also proved in \cite{pt} that the unitary group of 
the free group factor $\cL F_d$ of rank $d$ 
is contractible when $d=\infty$. 
Here we prove the same for all $d\geq2$. 
\begin{prop}[cf.\ Corollary $2$ in \cite{pt}]
Let $M\coloneqq M_0\otimes \cL F_d$ be the tensor product of 
a von Neumann algebra 
and the free group factor $\cL F_d$.  
Then $\cU(M)$ is contractible. 
\end{prop}
\begin{proof}
The proof is similar to that for Lemma~\ref{lem:bridge}.
For simplicity, we deal with the case 
$M=\cR*\cR\cong \cL F_2$ (\cite{dykema}). 
Let $\theta_t$ be as in Lemma~\ref{lem:car} and 
write $\cS\coloneqq\theta_1(\cR)$.
Take a Haar unitary element $w_0$ in $\cS'\cap\cR$.
We write $N\coloneqq\cS*\cS\subset M$. 
Then $\cU(M)$ contracts to $\cU(N)$ by Lemma~\ref{lem:car}.
Let $w_i$ denote the copy of $w_0$ 
in the $i$-th free component of $M=\cR*\cR$. 
Then $w\coloneqq w_1w_2\in\cU(M)$ is Haar. 
Moreover, since $\cU(\cS)w_0 = w_0\cU(\cS)$ 
consists of trace-zero elements, it is not hard to check 
by a direct computation that $w$ is free from $N$. 
It follows that $\cU(N)w$ consists entirely 
of Haar unitary elements. 
Hence $\cU(N)w$, and $\cU(N)$ as well, 
is contractible inside $\cU(M)$ by Lemma~\ref{lem:haar}. 
\end{proof}
%
%
%
\section{Proof of Theorem~\ref{thmB}}\label{sec:thmc}
Theorem~\ref{thmB} (in fact a slightly stronger form of it) 
is asserted in \cite{pt}. 
However, the Michael Selection Theorem (\cite{michael2}) 
is misquoted in the proof. 
Specifically, the geodesic structure given in 
Proof of Lemma~3 in p.96 in \cite{pt} does not meet 
the condition 5.1.(e) in \cite{michael2}. 
The condition 5.1.(e) roughly requires that 
if $(u_t)_{t\in[0,1]}$ and $(v_t)_{t\in[0,1]}$ are geodesic paths 
such that $\dist(u_0,v_0)<\ve$ and $\dist(u_1,v_1)<\delta\ll\ve$, 
then $\dist(u_t,v_t)<\ve$ for all $t\in[0,1]$. 
It is essential that $\ve>0$ appearing in the above 
are the same, because this condition is repeatedly used 
to form a convex structure. 
In this paper, we use more practical version of 
the Michael Selection Principle (\cite{michael1}) to 
fix this problem.

We first deal with the case where the inclusion 
$N\subset M$ of $\sigma$-finite 
von Neumann algebras is \emph{strongly stable}: 
$(N\subset M)\cong (N_0\otimes\cR\subset M_0\otimes\cR)$. 
We fix a faithful normal state 
$\omega\coloneqq\omega_0\otimes\tau$ on $M\coloneqq M_0\otimes\cR$ and 
work with the $2$-norm and the corresponding metric 
$\dist$ on $M$. 
To ease notation, write $G\coloneqq\cU(M)$, $H\coloneqq\cU(N)$, 
and $X\coloneqq H\backslash G$ with the quotient map $Q\colon G\to X$. 
Then $X$ is a (complete) metric space 
w.r.t.\ the induced metric. 
\textbf{NB!} We work with $H\backslash G$ instead of $G/H$ 
to make the following true:
\[
\dist_X(Hu,Hv)=\inf_{w,w'\in H}\|wu-w'v\|_2=\inf_{w\in H}\|u-wv\|_2=\dist_G(u, Hv).
\]

Let $(\theta_t)_{t\geq0}$ be as in Lemma~\ref{lem:car} 
and it still denote the endomorphism $\id\otimes\theta_t$ 
on $M$, which is $\omega$-preserving and hence is isometric w.r.t.\ $d$; 
and the same for $\phi_t$. 
\begin{lem}\label{lem:selection3}
Let a continuous map $\sigma\colon X\to G$ 
and $\ve>0$ be such that 
$\dist(\sigma(x),Q^{-1}(x))<\ve$ for all $x$. 
Then for every continuous function $s\colon X\to(0,\infty)$ 
and $\delta>0$, 
there is a continuous function $t\colon X\to (0,\infty)$ 
that satisfies $0<t(x)\le s(x)$ and 
\[
Q^{-1}(x)\cap\Ball(\sigma(x),\ve)
 \cap \{ u\in G : \| u - \theta_{t(x)}(u) \|_2<\delta\}
 \neq \emptyset
\]
for every $x\in X$.
\end{lem}
\begin{proof}
For each $t>0$, put $f_t(u)\coloneqq\max_{s\in[0,t]}\| u - \theta_s(u)\|_2$. 
Note that each $f_t$ is continuous on $G$ and that 
$f_t\searrow0$ pointwise on $G$ as $t\searrow0$.  
For each $t>0$, consider
\[
W_t\coloneqq\{ x\in X : Q^{-1}(x)\cap\Ball(\sigma(x),\ve)\cap\{ u\in G : f_t(u)<\delta\}
 \neq \emptyset \}.
\]

We claim that $\{ W_t \}_{t>0}$ is an open cover for $X$. 
That it covers $X$ is straightforward. 
To prove $W_t$ is open, let $x_0\in W_t$ be given. 
Then there is $\kappa\in (0,\ve)$ such that 
\[
Q^{-1}(x_0)\cap\Ball(\sigma(x_0),\ve-\kappa)\cap\{ u\in G : f_t(u)<\delta\}
 \neq \emptyset
\]
Hence
\[
Q(\Ball(\sigma(x_0),\ve-\kappa)\cap\{ u\in G : f_t(u)<\delta\})\cap\{ x\in X : \|\sigma(x_0)-\sigma(x)\|_2<\kappa\}
\]
is an open neighborhood of $x_0$ that is contained in $W_t$. 

Thus, there is a partition of unity $\{\psi_i\}$ for $X$ 
subordinated by $\{ W_t : t>0\}$.
For each $i$, take $t_i$ such that $\supp\psi_i\subset W_{t_i}$. 
Then the continuous function $t(x)\coloneqq(\sum_i \psi_i(x)t_i)\wedge s(x)$ satisfies the desired property. 
Indeed, for $r(x)\coloneqq\max\{ t_i : \psi_i(x)\neq0\}$ 
one has $t(x) \le r(x)$ and $x\in W_{r(x)}\subset W_{t(x)}$. 
\end{proof}

Here is the Michael Selection Principle (\cite{michael1}) as adapted to our setting. 
\begin{lem}\label{lem:pre}
Theorem~\ref{thmB} holds true if the inclusion $N\subset M$ 
is strongly stable.
\end{lem}
\begin{proof} 
We stick to the notation and 
work with the $2$-norm 
corresponding to a faithful normal state 
$\omega=\omega_0\otimes\tau$ on $M=M_0\otimes\cR$ 
and the quotient map $Q\colon G\to X\coloneqq H\backslash G$.  
Put $\delta_n\coloneqq 4\cdot 2^{-n}$. 
We will construct continuous maps $\vs_n\colon X\to G$ 
and $s_n\colon X\to(0,\infty)$ such that 
\begin{itemize}
\item
$\vs_n(x)\in\ran\theta_{s_n(x)}$ for every $x\in X$;
\item
$\dist(\vs_n(x),Q^{-1}(x)) < \delta_n$ for every $x\in X$;
\item
$\dist(\vs_{n-1}(x),\vs_{n}(x)) < \delta_{n-1}+\delta_n$ for every $x\in X$.
\end{itemize}
Once this is done, $\vs(x) \coloneqq \lim_n \vs_n(x)$ defines 
a continuous cross section for $Q$.

Put $\vs_0(x)\coloneqq 1$ and $s_0(x)\coloneqq 1$ for all $x\in X$. 
Fix $n$ and suppose that $\vs_{n-1}$ and $s_{n-1}$ are already constructed. 
Put $\gamma\coloneqq 1/4$ and take $t\colon X\to (0,\infty)$ 
as in Lemma~\ref{lem:selection3} for 
$\sigma=\vs_{n-1}$, $\ve=\delta_{n-1}$, $s(x)=s_{n-1}(x)$, and $\delta=\gamma\delta_n$. 
For every $u\in G$, consider 
\begin{align*}
W_u\coloneqq\{ x \in X : &\ Q^{-1}(x) \cap \Ball(\vs_{n-1}(x), \delta_{n-1}) 
 \cap \Ball(u, \gamma\delta_n) \neq\emptyset\\
 &\mbox{ and } \| u - \theta_{t(x)}(u)\|_2 < \gamma\delta_n\}.
\end{align*}

We claim that $\{W_u\}_{u \in G}$ is an open cover for $X$. 
That it covers $X$ is straightforward. 
To prove $W_u$ is open, let $x_0\in W_u$ be given. 
Then there is $\kappa\in(0,\gamma\delta_n)$ such that 
\[
Q^{-1}(x_0) \cap \Ball(\vs_{n-1}(x_0), \delta_{n-1}-\kappa) 
 \cap \Ball(u, \gamma\delta_n) \neq \emptyset
\]
and $\| u - \theta_{t(x_0)}(u)\|_2 < \gamma\delta_n-\kappa$. 
Hence for 
\begin{align*}
&Q(\Ball(\vs_{n-1}(x_0), \delta_{n-1}-\kappa)
 \cap \Ball(u, \gamma\delta_n))\\
 &\quad\cap \{ x\in X : 
  \|\vs_{n-1}(x_0)-\vs_{n-1}(x)\|_2 + \| \theta_{t(x_0)}(u) - \theta_{t(x)}(u)\|_2<\kappa\}
\end{align*}
is an open neighborhood of $x_0$ that is contained in $W_u$. 

Thus, there is a partition of unity $\{\psi_i\}_{i\in I}$ for $X$ 
subordinated by $\{ W_u\}$. 
We may assume that $I$ is totally ordered. 
For each $i$, take $u_i$ such that $\supp\psi_i\subset W_{u_i}$. 
We define $p_i(x)$ to be the orthogonal projection in $L^\infty[0,1]$ that 
corresponds to the interval $[\sum_{j<i}\psi_j(x), \sum_{j\le i}\psi_j(x))$. 
Note that the defining sum is a locally finite sum and that 
the maps $x\mapsto p_i(x)$ are ultrastrongly continuous. 
We put $s_n(x)\coloneqq t(x)/2$ and define 
a continuous map $\vs_n\colon X\to G$ by 
\[
\vs_n(x) \coloneqq \sum_i \theta_{t(x)}(u_i)\phi_{t(x)}(p_i(x)) \in \ran\theta_{s_n(x)}.
\]
Recall that $\phi_t$ is a continuous family of 
trace-preserving embeddings of $L^\infty[0,1]$ 
into $\theta_t(\cR)'\cap\theta_{t/2}(\cR)$ 
(as embedded in $M$). 
To prove $\vs_n$ satisfies the displayed conditions, 
let $x\in X$ be given and write $I(x)\coloneqq\{ i : \psi_i(x)>0\}$. 
For each $i\in I(x)$, take 
\[
v_i \in Q^{-1}(x) \cap \Ball(\vs_{n-1}(x), \delta_{n-1}) \cap \Ball(u_i, \gamma\delta_n) 
\]
and put $v\coloneqq\sum_i \theta_{t(x)}(v_i)\phi_{t(x)}(p_i(x))$. 
Then, 
\[
\| v- \vs_n(x) \|_2\le \max\{\| v_i-u_i\|_2 : i\in I(x)\}< \gamma\delta_n.
\]
Moreover, for a fixed $i_0\in I(x)$, one has  
$v_iv_{i_0}^{-1} \in H$ for all $i$ and 
hence $v\theta_{t(x)}(v_{i_0})^{-1}\in H$, because 
$\theta_t$ and $\phi_t$ keep $N$ invariant. 
Recall that $\| u_{i_0} - \theta_{t(x)}(u_{i_0})\|_2 < \gamma\delta_n$ because 
$x\in W_{i_0}$.
Thus, 
\[
\dist( \vs_n(x), Q^{-1}(x) ) \le \| \vs_n(x) - v \|_2 + \| \theta_{t(x)}(v_{i_0})-v_{i_0}\|_2
\le 4\cdot \gamma\delta_n=\delta_n.
\]
Also, since 
$\|  \vs_{n-1}(x) - u_i\|_2 < \delta_{n-1}+\gamma\delta_n$ and 
$\vs_{n-1}(x) \in \ran\theta_{s_{n-1}(x)}\subset \ran\theta_{t(x)}$ 
commutes with $\phi_{t(x)}(p_i(x))$ for all $i\in I(x)$,
one has 
\[
\|\vs_{n-1}(x)-\vs_n(x)\|_2<\max\{ \| \vs_{n-1}(x) - \theta_{t(x)}(u_i) \|_2 : i\in I(x) \}
 < \delta_{n-1}+2\cdot\gamma\delta_n.
\]
These altogether verify the required properties for $\vs_n$ and $s_n$. 
\end{proof}

\begin{proof}[Proof of Theorem~\ref{thmB}]
Let $N\subset M$ be $\sigma$-finite von Neumann algebras 
and now only assume that $N$ is strongly stable. 
As discussed in Section~\ref{sec:prelim2}, 
it suffices to prove existence of a $\cU(N)$-equivariant 
continuous retraction from $\cU(M)$ onto $\cU(N)$. 
Let $\cS$ denote a copy of $\cR$.
Since $N=N_0\otimes\cR$, we see that 
the inclusion 
\[
(N \otimes \IC1 \subset N\otimes\cS)
 \cong (N_0\otimes\IC1\otimes\cR\subset N_0\otimes\cS\otimes\cR)
\]
is strongly stable. 
Hence by Lemma~\ref{lem:pre}, $\cU(N)=\cU(N \otimes \IC1)$ 
is an equivariant retract of $\cU(N\otimes\cS)$. 
Since $\cU(N\otimes\cS)$ is an equivariant retract 
of $\cU(M\otimes\cS)$ by Lemma~\ref{lem:pre} again, 
we find a $\cU(N)$-equivariant retraction
\[
\cU(M) \hookrightarrow \cU(M\otimes\cS) \stackrel{\mathrm{retract}}{\to} \cU(N\otimes\cS)\stackrel{\mathrm{retract}}{\to}\cU(N).
\]
\vspace*{-1cm}

\end{proof}

\end{document}